\documentclass[11pt]{article}

\usepackage[utf8]{inputenc}
\usepackage[T1]{fontenc}
\usepackage[sc]{mathpazo}
\usepackage{microtype}
\usepackage{amsmath}
\usepackage{amssymb}
\usepackage{amsthm}
\usepackage{bm}
\usepackage{dsfont}
\usepackage{authblk}
\usepackage{fullpage}
\usepackage{comment}
\usepackage{tikz}
\usepackage{mathtools}
\usepackage[shortlabels]{enumitem}
\usepackage{complexity}
\usepackage[backend=biber, style=alphabetic, maxbibnames=99, url=false]{biblatex}
\usepackage{hyperref}
\usepackage[capitalize, nameinlink]{cleveref}
\usepackage{nag}

\newcommand{\declarecolor}[2]{\definecolor{#1}{RGB}{#2}\expandafter\newcommand\csname #1\endcsname[1]{\textcolor{#1}{##1}}}
\declarecolor{White}{255, 255, 255}
\declarecolor{Black}{0, 0, 0}
\declarecolor{LightGray}{216, 216, 216}
\declarecolor{Gray}{127, 127, 127}
\declarecolor{Orange}{237, 125, 49}
\declarecolor{LightOrange}{251,229, 214}
\declarecolor{Yellow}{255, 192, 0}
\declarecolor{LightYellow}{255, 242, 200}
\declarecolor{Blue}{91, 155, 213}
\declarecolor{LightBlue}{222, 235, 247}
\declarecolor{Green}{112, 173, 71}
\declarecolor{LightGreen}{226, 240, 217}
\declarecolor{Navy}{68, 114, 196}
\declarecolor{LightNavy}{218, 227, 243}

\hypersetup{
	colorlinks=true,
	pdfpagemode=UseNone,
	citecolor=Green,
	linkcolor=Navy,
	urlcolor=Black,
	pdfstartview=FitW
}

\newcommand{\declareperson}[1]{\expandafter\newcommand\csname#1\endcsname[1]{\textcolor{Orange}{#1: ##1}}}

\declareperson{Nima}
\declareperson{Kuikui}
\declareperson{Shayan}
\declareperson{Cynthia}

\theoremstyle{plain}
\newtheorem{theorem}{Theorem}[section]
\newtheorem{lemma}[theorem]{Lemma}
\newtheorem{corollary}[theorem]{Corollary}
\newtheorem{proposition}[theorem]{Proposition}

\newtheorem{conjecture}[theorem]{Conjecture}

\theoremstyle{definition}
\newtheorem{definition}[theorem]{Definition}

\theoremstyle{remark}
\newtheorem{remark}[theorem]{Remark}

\newlist{parts}{enumerate}{10}
\setlist[parts]{label=\arabic*.,ref=\arabic*}
\makeatletter
\if@cref@capitalise
	\crefname{partsi}{Part}{Parts}
\else
	\crefname{partsi}{part}{parts}
\fi
\makeatother
\Crefname{partsi}{Part}{Parts}

\newlist{forms}{enumerate}{10}
\setlist[forms]{label=\roman*),ref=(\roman*)}
\makeatletter
\if@cref@capitalise
	\crefname{formsi}{}{}
\else
	\crefname{formsi}{}{}
\fi
\makeatother
\Crefname{formsi}{}{}

\newlist{conds}{enumerate}{10}
\setlist[conds]{label=\rm{(\arabic*)},ref=\arabic*}
\makeatletter
\if@cref@capitalise
	\crefname{condsi}{}{}
\else
	\crefname{condsi}{}{}
\fi
\makeatother
\Crefname{condsi}{}{}

\usetikzlibrary{shapes, patterns, decorations, fit, intersections}

\newcommand*{\Z}{{\mathbb{Z}}}
\let\R\relax
\newcommand*{\R}{{\mathbb{R}}}

\newcommand*{\1}{{\mathds{1}}}
\newcommand*{\I}{{\mathcal{I}}}

\DeclareMathOperator{\rank}{rank}

\providecommand{\given}{}
\DeclarePairedDelimiterX{\card}[1]{\lvert}{\rvert}{\renewcommand\given{\nonscript\:\delimsize\vert\nonscript\:\mathopen{}}#1}
\DeclarePairedDelimiterX{\abs}[1]{\lvert}{\rvert}{\renewcommand\given{\nonscript\:\delimsize\vert\nonscript\:\mathopen{}}#1}
\DeclarePairedDelimiterX{\norm}[1]{\lVert}{\rVert}{\renewcommand\given{\nonscript\:\delimsize\vert\nonscript\:\mathopen{}}#1}
\DeclarePairedDelimiterX{\tuple}[1]{\lparen}{\rparen}{\renewcommand\given{\nonscript\:\delimsize\vert\nonscript\:\mathopen{}}#1}
\DeclarePairedDelimiterX{\parens}[1]{\lparen}{\rparen}{\renewcommand\given{\nonscript\:\delimsize\vert\nonscript\:\mathopen{}}#1}
\DeclarePairedDelimiterX{\brackets}[1]{\lbrack}{\rbrack}{\renewcommand\given{\nonscript\:\delimsize\vert\nonscript\:\mathopen{}}#1}
\DeclarePairedDelimiterX{\set}[1]\{\}{\renewcommand\given{\nonscript\:\delimsize\vert\nonscript\:\mathopen{}}#1}
\let\Pr\relax
\DeclarePairedDelimiterXPP{\Pr}[1]{\mathbb{P}}[]{}{\renewcommand\given{\nonscript\:\delimsize\vert\nonscript\:\mathopen{}}#1}
\DeclarePairedDelimiterXPP{\PrX}[2]{\mathbb{P}_{#1}}[]{}{\renewcommand\given{\nonscript\:\delimsize\vert\nonscript\:\mathopen{}}#2}
\DeclarePairedDelimiterXPP{\Ex}[1]{\mathbb{E}}[]{}{\renewcommand\given{\nonscript\:\delimsize\vert\nonscript\:\mathopen{}}#1}
\DeclarePairedDelimiterXPP{\ExX}[2]{\mathbb{E}_{#1}}[]{}{\renewcommand\given{\nonscript\:\delimsize\vert\nonscript\:\mathopen{}}#2}

\newcommand*{\eval}[1]{\left.#1\right\rvert}

\title{Log-Concave Polynomials III: 
Mason's Ultra-Log-Concavity Conjecture for Independent Sets of Matroids}

\author{Nima Anari}
\affil{\small Stanford University, \textsf{anari@cs.stanford.edu}}

\author{Kuikui Liu}
\author{Shayan Oveis Gharan}
\affil{\small University of Washington, \textsf{liukui17@cs.washington.edu}, \textsf{shayan@cs.washington.edu}}

\author{Cynthia Vinzant}
\affil{\small North Carolina State University, \textsf{clvinzan@ncsu.edu}}

\bibliography{refs}

\begin{document}
	\maketitle
\begin{abstract}
We give a self-contained proof of the strongest version of Mason's conjecture, namely that for any matroid the sequence of the number of independent sets of given sizes is ultra log-concave. To do this, we introduce a class of polynomials, called completely log-concave polynomials, whose bivariate restrictions have ultra log-concave coefficients. At the heart of our proof we show that for any matroid, the homogenization of the generating polynomial of its independent sets is completely log-concave.
\end{abstract}	

\section{Introduction}
\label{sec:intro}

Matroids are combinatorial structures that model various types of independence, such as linear independence of vectors in a linear space or algebraic independence of  elements in a field extension. For an inspiring recent survey, see \cite{Ard18}. There have been several recent breakthroughs proving inequalities on sequences of numbers associated to matroids.  While the proofs in this paper are self-contained, we build off several of these ideas to study the following conjecture of Mason \cite{Mas72}.
\begin{conjecture}[Mason's Conjecture]\label{conj:Mason}
For an $n$-element matroid $M$ with $\I_k$ independent sets of size $k$, 
\begin{forms}
    \item\label{mason:i} $\I_{k}^{2} \geq \I_{k-1} \cdot \I_{k+1}$ (log-concavity),
    \item\label{mason:ii} $\I_{k}^{2} \geq \left(1 + \frac{1}{k}\right) \cdot \I_{k-1} \cdot \I_{k+1}$,
    \item\label{mason:iii} $\I_{k}^{2} \geq \left(1 + \frac{1}{k}\right) \cdot \left(1 + \frac{1}{n-k}\right) \cdot \I_{k-1} \cdot \I_{k+1}$ (ultra log-concavity).
\end{forms}
\end{conjecture}
Note that \cref{mason:i,mason:ii,mason:iii} are written in increasing strength. \Textcite{AHK18} proved \cref{mason:i} using techniques from Hodge theory and algebraic geometry. Building on this, \textcite{HSW18} proved \cref{mason:ii}. Prior to our work, \cref{mason:iii} was only proven to hold when $n \leq 11$ or $k\leq 5$ \cite{KN11}. We refer to \cite{Sey75, Dow80, Mah85, Zha85, HK12, HS89, Len13} for other partial results on Mason's conjecture. Here, we give a self-contained proof of \cref{mason:iii}.

\begin{theorem}\label{thm:mason}
	For a matroid $M$ on $n$ elements with $\I_k$ independent sets of size $k$, 	
	the sequence  $\I_{0}, \I_{1},\dots,\I_{n}$ is ultra log-concave. That is, 
	 for $1<k<n$, 
		\[ \parens*{\frac{\I_{k}}{\binom{n}{k}}}^2 \ \ \geq  \ \ \frac{\I_{k-1}}{\binom{n}{k-1}} \cdot \frac{\I_{k+1}}{\binom{n}{k+1}} . \]
\end{theorem}

We prove \cref{thm:mason} in \cref{sec:mason}.
The main tool we use will be polynomials that are log-concave as functions on the positive orthant. 
For $i\in [n]$, let $\partial_i$ or $\partial_{z_i}$ denote the partial derivative operator that maps a polynomial $f$ to its partial derivative with respect to $z_i$. For a vector $v\in \R^n$, we let $D_v$ denote the directional derivative operator in direction $v$,
\[ D_v=\sum_{i=1}^n v_i \partial_i. \]
We call a polynomial $f\in \R[z_1,\dots,z_n]$ log-concave over $\R_{\geq 0}^n$ if $f$ is nonnegative and log-concave
 as a function over $\R_{\geq 0}^n$, or in other words if for every $u, v\in \R_{\geq 0}^n$ and $\lambda\in [0,1]$, we have $f(u), f(v)\geq 0$ and
\[ f(\lambda u+(1-\lambda)v)\geq f(u)^\lambda \cdot f(v)^{1-\lambda}. \]
Note that the zero polynomial is also log-concave. 
If $f(v)$ is positive for some $v\in \R_{\geq 0}^n$, 
then we call $f$ log-concave at $z=v$ if the Hessian of its $\log$ at $v$ is negative semidefinite. 
It is easy to see from the definition that for any fixed $d$ and $n$, the set of
 polynomials of degree at most $d$ in $n$ variables that are log-concave on $\R_{\geq 0}^n$ is closed in the Euclidean topology on $\R[x_1, \hdots, x_n]_{\leq d}$.  Also, a nonzero polynomial is log-concave over $\R_{\geq 0}^n$ 
 if and only if it is log-concave at every point of $\R_{>0}^n$.
\begin{definition}\label{def:clc}
A polynomial $f\in \R[z_1,\dots,z_n]$ is \textbf{completely log-concave} if for every set of nonnegative vectors $v_1,\dots,v_k\in \R_{\geq 0}^n$, the polynomial $D_{v_1}\dots D_{v_k} f$ is nonnegative and log-concave over $\R_{\geq 0}^n$. 
\end{definition}
Completely log-concave polynomials were introduced in \cite{AOV18} based on similar notions of strongly log-concave and Alexandrov-Fenchel polynomials first studied in \cite{Gur08}. 
In this paper, we prove the properties of complete log-concavity necessary for \cref{thm:mason} and defer a more detailed treatment of completely log-concave polynomials to a future article.

The main ingredient of the proof of \cref{thm:mason} is to show that the homogenization of the generating polynomial of all independent sets of $M$ is completely log-concave, namely that the polynomial 
$$ g_M(y,z_1,\dots,z_n)=\sum_{I\in\I} y^{n-\card{I}}\prod_{i\in I} z_i$$
is completely log-concave. Then, we use this to show that the bivariate restriction $f_M(y,z)=\sum_{k=0}^r \I_{k} y^{n-k}z^k$ is completely log-concave.
Finally, we derive \cref{thm:mason} from the latter fact based on an observation of Gurvits \cite{Gur08} 
on the coefficients of completely log-concave polynomials.

\subsection{Independent work} 
In a related upcoming work, Br\"and\'en and Huh have independently developed methods that overlap with our work. In particular they also prove the strongest version of Mason's conjecture. 

\subsection{Spectral negative dependence}
It is well-known that the uniform distribution over all spanning trees of a graph is negatively correlated and more generally negatively associated, see \cite{Pem00} for background. This fact more generally extends to regular matroids. Prior to our work many researchers tried to approach  Mason's conjecture through the lens of negative correlation \cite{SW75, Wag08, BBL09, KN10, KN11}. However, for many matroids the uniform distribution on bases is not negatively correlated and furthermore, negative correlation does not necessarily imply log-concavity of its rank sequences \cite{Wag08}.

Consider the polynomial $p_M=\sum_{B} \prod_{i\in B} z_i$, where the sum is over all bases of the matroid $M$. Then the negative correlation property is equivalent to all off-diagonal entries of the Hessian of $\log p_M$ being non-positive when evaluated at the 
all-ones vector $\1 = (1,\hdots, 1)$, i.e. 
$$ (\nabla^2 \log p_M(\1))_{i,j} = p_M(\1)\cdot \partial_i\partial_j p(\1)- \partial_i p_M(\1)\cdot\partial_j p_M(\1) \leq 0,$$
for all $1\leq i,j\leq n$, $i\neq j$.
This inequality holds for regular matroids but not necessarily for even linear matroids. 

In \cite{AOV18} it was observed that for any matroid $M$, the polynomial $p_M$ is completely log-concave. This means that even though $\nabla^2 \log p_M(\1)$ can have positive entries, all of its eigenvalues, and eigenvalues of Hessian of the $\log$ of all partials of $p_M$, are non-positive. We call this property, \emph{spectral negative dependence}. In this paper, we show that for any matroid, the homogenization of the generating polynomial of all independent sets, namely $g_M$ also satisfies spectral negative dependence. Furthermore, spectral negative dependence is enough to prove the strong form of log-concavity of rank sequences as conjectured by Mason.

\paragraph{Acknowledgements.} 
Part of this work was started while the first and last authors were visiting the Simons Institute for the Theory of Computing. It was partially supported by the DIMACS/Simons Collaboration on Bridging Continuous and Discrete Optimization through NSF grant CCF-1740425. Shayan Oveis Gharan and Kuikui Liu are supported by the NSF grant CCF-1552097 and ONR-YIP grant N00014-17-1-2429. Cynthia Vinzant was partially supported by the National Science Foundation grant DMS-1620014.

\section{Preliminaries}\label{sec:prelim}

A polynomial $f\in \R[z_1,\dots,z_n]$ is homogeneous of degree $d$ if every monomial of $f$ has degree $d$, or equivalently $f(\lambda\cdot z_1,\dots,\lambda\cdot z_n)=\lambda^d f(z_1,\dots,z_n)$ for all $\lambda \in \R$. We will use $\nabla f$ to denote the gradient of $f$ and $\nabla^2 f$ to denote its Hessian matrix. 

We use $[n]$ to refer to $\set{1,\dots,n}$. When $n$ is clear from context, for a set $S\subseteq [n]$, we let $\1_S\in \R^n$ denote the indicator vector of $S$. For variables $z_{1},\dots,z_{n}$ and $S \subseteq [n]$, we let $z^{S} $ denote the monomial $\prod_{i \in S} z_{i}$. Similarly, for an integer vector $\alpha = (\alpha_1, \hdots, \alpha_n) \in \Z_{\geq 0}^n$ or a subset $S\subseteq [n]$, we denote differential operators 
\[\partial^{\alpha}=\prod_{i=1}^n\partial_i^{\alpha_i}   \ \ \text{ and } \ \  \partial^{S} =\partial^{\1_{S}}=\prod_{i\in S} \partial_i.  \]
Note that if $f$ is homogeneous of degree $d$, then $\partial^{\alpha}f$ is homogenous of degree $d-\abs{\alpha}$ where $\abs{\alpha} = \sum_{i=1}^n\alpha_i$.

A symmetric matrix $Q\in\R^{n\times n}$, alternatively viewed as a quadratic form $z\mapsto z^\intercal Q z$, is positive semidefinite if $v^\intercal Q v\geq 0$ for all $v\in\R^n$ and negative semidefinite if $v^\intercal Q v\leq 0$ for all $v\in\R^n$. If these inequalities are strict for $v\neq 0$, then $Q$ is positive or negative definite, respectively. There are several equivalent definitions. In particular, a matrix is positive semidefinite if and only if all of its eigenvalues are nonnegative, which occurs if and only if the all its principal minors are nonnegative. 
Since $Q$ is negative semidefinite if and only if $-Q$ is positive semidefinite, these translate into analogous characterizations of negative semidefinite-ness. 

\subsection{Matroids}

Formally, a \textbf{matroid} $M=([n],\I)$ consists of a ground set $[n]$ and a nonempty collection $\I$ of \emph{independent} subsets of $[n]$ satisfying the following two conditions: 
\begin{conds}
	\item If $S\subseteq T$ and $T\in \I$, then $S\in \I$.
 	\item If $S,T\in \I$ and $\card{T}>\card{S}$, then there exists an element $i\in T \setminus S$ such that $S\cup \set{i}\in \I$.
\end{conds}

The \textbf{rank}, denoted by $\rank(S)$, of a subset $S\subseteq [n]$ is the size of the largest independent set contained in $S$ 
and the rank of $M$ is defined as $\rank([n])$.  An element $i\in [n]$ is called a \textbf{loop} if $\set{i}\notin \I$, and two 
elements $i,j\in [n]$ are called \text{parallel} if neither is a loop and $\rank(\set{i,j}) = 1$. One can check that parallelism defines an equivalence relation on the non-loops of $M$, which partitions the set of non-loops into parallelism classes. 

For a matroid $M$ and an independent set $S\in \I$, the \textbf{contraction}, $M/S$, of $M$ by $S$ is the matroid on ground set $[n]\setminus S$ with independent sets $\set{T \subseteq [n]\setminus S \given  S\cup T\in \I}$. 
In particular, the rank of $M/S$ equals $\rank(M) - \card{S}$. 
See \cite{Oxl11} for more details and general reference. 

\subsection{Log-concave polynomials}
In  \cite{AOV18}, it was shown that a homogeneous polynomial $f \in \R[z_1,\dots,z_n]$ with nonnegative coefficients is log-concave at a point $z=a$ if and only if its Hessian $\nabla^2 f$ has at most one positive eigenvalue at $z=a$. 
One can relate this directly to the negative semidefinite-ness of the Hessian of $\log(f)$. Indeed, there are several useful equivalent characterizations of this condition:

\begin{lemma}\label{lem:LC_Quad}
Let $f\in \R[z_1, \hdots, z_n]$ be homogeneous of degree $d\geq 2$ with nonnegative coefficients. Fix a point $a\in \R_{\geq 0}^n$ with $f(a)\neq 0$, and let 
$Q = \eval{\nabla^2 f}_{z=a}$. 
The following are equivalent:
\begin{conds}
\item\label{cond:1} $f$ is log-concave at $z=a$,
\item\label{cond:2} $z\mapsto z^{\intercal}Qz$ is negative semidefinite on $(Qa)^{\perp}$,
\item\label{cond:3} $z\mapsto z^{\intercal}Qz$ is negative semidefinite on $(Qb)^{\perp}$ for every $b\in\R^n_{\geq 0}$ such that $Qb\neq 0$, 
\item\label{cond:4} $z\mapsto z^{\intercal}Qz$ is negative semidefinite on some linear space of dimension $n-1$, and
\item\label{cond:5} the matrix $(a^{\intercal}Qa) Q - (Qa)(Qa)^{\intercal}$ is negative semidefinite.
\end{conds}
For $d\geq 3$, these are also equivalent to the condition 
\begin{conds}[resume]
\item\label{cond:6} $D_af$ is log-concave at $z=a$. 
\end{conds}
\end{lemma}

One can check that this condition is also equivalent to $Q$ having at most one positive eigenvalue, but we do not rely on 
this fact and leave its proof to the interested reader.

\begin{proof} 
Euler's identity states that for a homogeneous polynomial $g$ of degree $d$, $\sum_{i=1}^nz_i\partial_{i}g$ equals $d\cdot g$. 
Using this on $f$ and $\partial_{j}f$ gives that $ Qa=(d-1)\cdot\nabla f(a)$ and $a^\intercal Qa= d(d-1)\cdot f(a) $. 
The Hessian of $\log(f)$ at $z=a$ then equals 
\[
\eval{\nabla^2(\log(f))}_{z=a} \  = \
\eval{\parens*{\frac{ f \cdot \nabla^2f-  \nabla f\nabla f^\intercal}{f^2}}}_{z=a} \
 = \  
d(d-1)\frac{a^\intercal Qa  \cdot Q-  \frac{d}{d-1} (Qa)(Qa)^\intercal  }{(a^{\intercal}Qa)^2}.
\]
We can also conclude that $a^\intercal Qa = d(d-1)\cdot f(a)>0$ and that the vector $Qa$ is nonzero. 

(\cref{cond:1} $\Rightarrow$ \cref{cond:2})  
 If $f$ is log-concave at $z=a$, then the Hessian of $\log(f(z))$ at $z=a$ is negative semidefinite. 
Restricted to the linear space $(Qa)^{\perp} = \set{z\in \R^n \given z^{\intercal}Qa=0}$, the formula above simplifies to $\frac{ d(d-1)}{a^\intercal Qa} \cdot Q$, meaning that 
 $z\mapsto z^{\intercal}Qz$ is negative semidefinite on this linear space.

(\cref{cond:2} $\Rightarrow$ \cref{cond:4})  Since $Qa$ is nonzero, $(Qa)^{\perp}$ has dimension $n-1$. 

(\cref{cond:4} $\Rightarrow$ \cref{cond:5}) Suppose that $z\mapsto z^{\intercal}Qz$ is negative semidefinite on an $(n-1)$-dimensional linear space $L$. 
Let $b\in \R^n$ and consider the $n\times 2$ matrix $P$ with columns $a$ and $b$. Then 
	\[
		P^\intercal Q P=
		\begin{bmatrix} 
			a^\intercal Qa & a^\intercal Qb \\ 
			b^\intercal Qa & b^\intercal Qb
		\end{bmatrix}.
	\]
If $P$ has rank one, then so does $P^\intercal Q P$, meaning that $\det(P^\intercal Q P)=0$. 
Otherwise $P$ has rank two and its column-span intersects $L$ nontrivially. 
This means there is a vector $v\in \R^2$
for which $Pv \in L$ is nonzero and $(Pv)^{\intercal}Q(Pv) \leq 0$. 
From this we see that $P^\intercal Q P$ is not positive definite. 
On the other hand, since the diagonal entry $a^\intercal Qa$ is positive, 
$P^\intercal Q P$ is not negative definite. 
In either case, we then find that 
\[		\det(P^\intercal Q P)=(a^\intercal Q a) \cdot(b^\intercal Q b)  - (b^\intercal Qa) \cdot (a^\intercal Qb)\leq 0.
\]
Thus $b^\intercal ((a^\intercal Q a) \cdot Q - (Qa)(Qa)^\intercal) b \leq 0$ for all $b\in \R^n$. 	

(\cref{cond:5}  $\Rightarrow$ \cref{cond:1}) Suppose $(a^{\intercal}Qa) \cdot Q -  (Qa)(Qa)^{\intercal}$ is negative semidefinite. 
Further subtracting $\frac{1}{d-1}(Qa)(Qa)^{\intercal}$ 
and scaling by the positive number $\frac{d(d-1)}{(a^{\intercal}Qa)^2}$ results in $\eval{\nabla^2(\log(f))}_{z=a} $,  as above, which must therefore also be negative semidefinite. 

(\cref{cond:3} $\Leftrightarrow$ \cref{cond:4}) Both conditions depend only on the matrix $Q$. We can then use the equivalence (\cref{cond:2} $\Leftrightarrow$ \cref{cond:3}) for the point $z=b$ and the quadratic polynomial $f(z) = \frac{1}{2}  z^\intercal Q z$, whose Hessian at any point is the matrix $Q$.

(\cref{cond:1} $\Leftrightarrow$ \cref{cond:6}) For $d\geq 3$,  $D_af$ is homogeneous of degree $\geq 2$. 
Euler's identity applied to $\partial_{i}\partial_{j}f$ shows that 
the Hessian of $D_af$ at $z=a$ is a scalar multiple of the Hessian of $f$ at $z=a$, namely $(d-2)\eval{\nabla^2f}_{z=a}$. 
Thus by the equivalence (\cref{cond:1} $\Leftrightarrow$ \cref{cond:4}), $D_af$ is log-concave at $a$ if and only if $f$ is.
\end{proof}

\subsection{Completely log-concave polynomials}
One of the basic operations that preserves complete log-concavity is 
an affine change of coordinates. 
This was first proved in \cite{AOV18}, but for completeness we include the proof here.
\begin{lemma}\label{lem:LinearPreserver}
If $f\in \R[z_1,\dots,z_n]$ is completely log-concave and $T:\R^m\to \R^n$ is an affine transform such that $T(\R_{\geq 0}^m)\subseteq \R_{\geq 0}^n$, then $f(T(y_1,\dots,y_m))\in \R[y_1,\dots,y_m]$ is completely log-concave.
\end{lemma}
\begin{proof}
    First, we prove that if $f$ is a log-concave polynomial, then $f \circ T = f(T(y_1,\dots,y_m))$ is also log-concave. 
    By assumption  for any $u,v \in \R_{\geq0}^{m}$, we have $T(u), T(v)\in \R_{\geq 0}^n$. Thus for any $0 \leq \lambda \leq 1$,
	\[ f(T(\lambda u+(1-\lambda)v)) \ = \ f(\lambda T(u)+(1-\lambda)T(v)) \ \geq  \ f(T(u))^\lambda f(T(v))^{1-\lambda}. \]
	Therefore $f \circ T$ is log-concave.
	
	Now suppose that $f$ is completely log-concave and let $v_{1},\dots,v_{k} \in \R_{\geq0}^{m}$. 
	Since $T(\R_{\geq0}^{m}) \subseteq \R_{\geq0}^{n}$ and $T$ is affine,
	$T(x) = Ax + b$ for some $A \in \R_{\geq0}^{n \times m}$ and $b \in \R_{\geq 0}^{n}$. 
	In particular, $Av_1,\dots,Av_k\in \R_{\geq 0}^n$, 
	which means that $D_{Av_1}\dots D_{Av_k}f$ is log-concave over $\R_{\geq 0}^n$.
	By the chain rule for differentiation, we have
	\[ D_{v_1}\dots D_{v_k}(f\circ T)=(D_{Av_1}\dots D_{Av_k} f)\circ T. \]
	 Since composition with $T$ preserves log-concavity, this polynomial is log-concave over $\R_{\geq 0}^m$.
	 \end{proof}

\section{Reduction to quadratics}
\label{sec:clc}
As the main result of this section we will show that, under some mild restrictions, 
to check whether a homogeneous polynomial is completely log-concave, 
it suffices to check the conditions in \cref{def:clc} for $k=d-2$ and $v_1,\dots,v_k\in \set{ \1_{\set{1}},\dots,\1_{\set{n}}}$. 
Then $D_{v_1}\cdots D_{v_k}f$ has the form $\partial^{\alpha}f$ where $\alpha_j$ is the number of vectors $v_k$ equal to $\1_{\set{j}}$. 
This provides a powerful tool to check complete log-concavity. The mild restriction we impose is \emph{indecomposability} of $f$ and its derivatives. 

\begin{definition}\label{def:indecomposable}
	A polynomial $f\in \R[z_1,\dots,z_n]$ is \textbf{indecomposable} if it cannot be written as $f_1+f_2$, where $f_1,f_2$ are 
	nonzero polynomials in disjoint sets of variables. Equivalently, if we form a graph with vertices $\set{i\given \partial_i f\neq 0}$ and edges $\set{(i, j)\given \partial_i\partial_j f\neq 0}$, then $f$ is indecomposable if and only if this graph is connected.
\end{definition}

Now we are ready to state the main result of this section.
\begin{theorem}\label{thm:CLCQuad}
Let $f$ be a homogeneous polynomial $f\in \R[z_1, \hdots, z_n]$ of degree $d \geq 2$ with nonnegative coefficients. 
If the following two conditions hold, then $f$ is completely log-concave:
\begin{forms}
\item For all $\alpha\in \Z_{\geq 0}^n$ with $\abs{\alpha}\leq d-2$, the polynomial $\partial^{\alpha}f$ is indecomposable.
\item For all $\alpha\in \Z_{\geq 0}^n$ with $\abs{\alpha}= d-2$, the quadratic polynomial $\partial^{\alpha}f$ is log-concave over $\R_{\geq 0}^n$. 	
\end{forms}
\end{theorem}
The converse of the above statement is also true, namely, every completely polynomial is indecomposable, but 
we defer the proof of this fact to a future article.

We build up to the proof of this theorem with a series of lemmas. 
The first is a criterion for the sum of two log-concave polynomials to be log-concave. We will then use this to prove that if a polynomial $f$ is indecomposable and all of its partial derivatives $\partial_i f$ are log-concave, then it itself must be log-concave. The proof of \cref{thm:CLCQuad} then follows by an induction on the degree.

\begin{lemma}\label{lem:Add}
Let $f,g\in \R[z_1,\hdots, z_n]$ be homogenous with nonnegative coefficients satisfying $D_b f = D_cg \neq 0$ for some vectors $b,c\in \R_{\geq 0}^n$. 
 If $f$ and $g$ are log-concave on $\R_{\geq 0}^n$ then so is $f+g$. 
\end{lemma}
\begin{proof}
The assumption that $D_b f = D_cg \neq 0$ means that $f$ and $g$ have the same degree $d$. We proceed by induction on $d$. 
If $d = 1$, then $f+g$ is a linear form with nonnegative coefficients, which is automatically log-concave on $\R_{\geq 0}^n$.
Now suppose $d\geq 2$. 
Fix $a\in \R_{> 0}^n$ and let $Q_1 = \nabla^2f(a)$ and $Q_2 = \nabla^2g(a)$. 
Then $D_b f = D_cg$ implies that for each $i=1, \hdots, n$, 
\[
(Q_1b)_i \ = \ \eval{( \partial_{i}D_b f)}_{z=a} \  =  \ \eval{( \partial_{i}D_c g)}_{z=a} \  =  \ (Q_2 c)_i,
\]
showing that $Q_1b = Q_2 c$.  Since $D_b f$ has nonnegative coefficients and is not identically zero, we also have 
that $D_b f(a) \neq 0$, meaning that $Q_1b\neq 0$.  By \cref{lem:LC_Quad} (\cref{cond:1} $\Rightarrow$ \cref{cond:3}) and the log-concavity of $f$ and $g$,
each quadratic form $z\mapsto z^{\intercal}Q_iz^{\intercal}$ is negative semidefinite on $(Q_1b)^{\perp}=(Q_2c)^{\perp}$. 
It follows that their sum $z\mapsto z^{\intercal}(Q_1+Q_2)z$ 
 given by the matrix $Q_1 + Q_2 = \eval{\nabla^2(f+g)}_{z=a}$ is also negative semidefinite on 
this $(n-1)$-dimensional linear space, so by \cref{lem:LC_Quad} (\cref{cond:4} $\Rightarrow$ \cref{cond:1}), $f+g$ is log-concave at $z=a$. 
\end{proof}

\begin{lemma}\label{lem:DerToCLC}
Let $f\in \R[z_1, \hdots, z_n]$ be homogeneous of degree $d\geq 3$ and indecomposable with nonnegative coefficients. 
If $\partial_i f$ is log-concave on $\R_{\geq 0}^n$ for every $i=1, \hdots, n$, 
then so is $D_af$ for every $a\in \R_{\geq 0}^n$. 
\end{lemma}

\begin{proof}
If $\partial_{i}f$ is identically zero for some $i$, then we can consider $f$ as a polynomial in the other variables. 
Without loss of generality, we can assume that $\partial_{i}f$  is nonzero for all $i$, and 
if necessary relabel $z_1, \hdots, z_n$ so that 
that for every $2\leq j \leq n$, there exists $i<j$ for which $\partial_{i} \partial_{j} f$ is non-zero. The latter follows from indecomposability.

Fix $a\in \R_{> 0}^n$. We will show that $D_af$ is log-concave on $\R_{\geq 0}^n$.
We show by induction on $k$ that for any $1\leq k\leq n$, $\sum_{i=1}^k a_i \partial_{i} f$
is  log-concave on $\R_{\geq 0}^n$. The case $k=1$ follows by assumption. 
For $1\leq k <n$, let $b$ denote the truncation of $a$ to its first $k$ coordinates, $b = (a_1, \hdots, a_k, 0, \hdots, 0)$ and let $c$ denote the vector $a_{k+1}\1_{\set{k+1}}$. 
By induction both $D_b f$ and $D_c f$ are log-concave, and 
\[
D_c D_b f  \ = \  D_bD_c f \ = \ \sum_{i=1}^k a_i a_{k+1} \partial_{i} \partial_{{k+1}} f.
\]
Since the coefficients of each summand are nonnegative and  $\partial_{i} \partial_{{k+1}} f$ is non-zero for some $1\leq i \leq k$, this sum is also non-zero. 
Then by \cref{lem:Add}, $D_b f + D_c f = \sum_{i=1}^{k+1} a_i \partial_{i} f$ is  log-concave on $\R_{\geq 0}^n$. 
For $k=n-1$, this is exactly $D_af$. 
Taking closures then shows that  $D_af$ is log-concave on $\R_{\geq 0}^n$ for all $a\in \R_{\geq 0}^n$.\end{proof}

\begin{proof}[Proof of \cref{thm:CLCQuad}]
We induct on $d = \deg(f)$. The case $d=2$ is clear, so let $d\geq 3$.  
For any positive vector $v\in \R_{>0}^n$, $D_vf$ is also indecomposable.  
Indeed for any homogeneous polynomial $g$ of degree $\geq 1$
with nonnegative coefficients (such as $\partial_{i}f $ and $\partial_{i}\partial_{j}f $), $D_vg $ is identically zero if and only if $g$ is.  

By taking closure, it suffices to show that for vectors $v_1, \hdots, v_k\in \R_{>0}^n$, the polynomial 
$D_{v_1}\cdots D_{v_k}f$ is log-concave on $\R_{\geq 0}^n$.
 If $k\geq d-1$, then $D_{v_1}\cdots D_{v_k}f$ is either identically zero or linear with nonnegative coefficients, 
in which case it is log-concave on $\R_{\geq 0}^n$, so we take  $0\leq k\leq d-2$. 
If $k=0$, then to show that $f$ is  log-concave at a point $a\in \R_{\geq 0}^n$, by \cref{lem:LC_Quad} (\cref{cond:6} $\Rightarrow$ \cref{cond:1}), 
it suffices to show that $D_af$ is log-concave at $z=a$. This reduces the case $k=0$ to the case $k=1$. 

Suppose $1\leq k \leq d-2$. By induction $\partial_{j}f$ is completely log-concave for all $j=1, \hdots, n$, 
and hence $D_{v_1}\cdots D_{v_{k-1}}\partial_{j}f = \partial_{j}D_{v_1}\cdots D_{v_{k-1}}f$ is log-concave on $\R_{\geq 0}^n$.  
Since $D_{v_1}\cdots D_{v_{k-1}}f$ is indecomposable and has degree $d-k+1\geq 3$, it follows from \cref{lem:DerToCLC} that 
$D_{v_1}\cdots D_{v_{k}}f$ is log-concave on $\R_{\geq 0}^n$.  
\end{proof}

\section{Complete log-concavity of independence polynomials}
\label{sec:IndPoly}
In this section, we use \cref{thm:CLCQuad} to prove
 that the homogenization of the generating polynomial of the independent sets of a 
matroid is completely log-concave. In the following section we use a restriction of this to derive Mason's conjecture. 

	\begin{theorem}\label{thm:IndPolyCLC} 
	For any matroid $M = ([n],\I)$, the polynomial 
		\[ g_M(y,z_1, \hdots, z_n) \ = \ \sum_{I\in \I} y^{n-\card{I}}\prod_{i\in I}z_i \]
	in $\R[y,z_1, \hdots, z_n]$  is completely log-concave. 
	\end{theorem}

We prove this by looking at quadratic derivatives of $g_M$. 

	\begin{lemma}\label{lem:Rank2}
For any matroid $M = ([n],\I)$, the quadratic polynomial $\partial_{y}^{n-2}g_M$
is log-concave on $\R_{\geq 0}^{n+1}$. 
\end{lemma}

\begin{proof}
After taking derivatives and rescaling, we see that 
\[ q \ = \ \frac{\partial_{y}^{n-2}g_M }{(n-2)!} \ = \   \frac{n(n-1)}{2}\cdot y^2 +(n-1)\cdot \sum_{\set{i}\in \I}y z_i+ \sum_{\set{i, j}\in \I} z_iz_j. \]

Let $Q$ denote the Hessian $\nabla^2q$ of $q$.  
Note that columns and rows of $\nabla^2 q$ corresponding to loops in $M$ are zero, and the log-concavity 
of $q$ only depends on the principal submatrix of $Q$ indexed by non-loops.   
In this spirit and in a slight abuse of notation, we use $\1$ within this proof to denote the indicator vector of the non-loops of $M$. 
Then we find that 
	\[ \nabla^2q   \ \ = \ \  Q \ \ = \ \ 
	\begin{bmatrix}
			n(n-1)&  (n-1) \1^\intercal\\
			(n-1) \1& B
	\end{bmatrix},
	 \]
	where $B$ is an $n\times n$ matrix with $B_{ij}=1$ when $\set{i,j}$ has rank two in $M$ and $B_{ij}=0$ otherwise. 
	Since $q$ is quadratic, its Hessian does not depend on any evaluation, so $q$ is log-concave on $\R_{\geq 0}^{n+1}$ if and only if it is 
	log-concave at the point $a = (1,0, \hdots, 0)$. By \cref{lem:LC_Quad} (\cref{cond:1} $\Leftrightarrow$ \cref{cond:5}), this happens if and only if the matrix 	
		\[
	(a^{\intercal}Qa)Q - (Qa)(Qa)^{\intercal}  
	\ = \ n(n-1)Q - (n-1)^2 \begin{bmatrix} n \\ \1\end{bmatrix}\begin{bmatrix} n \\ \1\end{bmatrix}^{\intercal}  
	\ = \ (n-1)\begin{bmatrix}
			0& 0\\
			0& n B - (n-1)\1\1^{\intercal}
	\end{bmatrix}\]	
	is negative semidefinite. Thus it suffices to show that $n B - (n-1)\1\1^{\intercal}$ is negative semidefinite. 
		As $M$ is a matroid, the matroid partition property tells us that the nonloops of $M$ may be partitioned into equivalence classes of parallel elements $P_{1},\dots,P_{c}$.
	This lets us rewrite the matrix $B$ as 
		\[ 
		B \ = \ \1\1^\intercal-\sum_{i=1}^c \1^{\ }_{P_i} \1_{P_i}^\intercal 
		\ \ \ \ \text{ and }\ \ \ \
		n B - (n-1)\1\1^{\intercal}  \ = \   \1\1^\intercal- n\cdot \sum_{i=1}^c \1^{\ }_{P_i} \1_{P_i}^\intercal.
		\]
	We can now check that this matrix is negative semidefinite. Let $x\in \R^n$ and consider 
	\[
	x^\intercal(n B - (n-1)\1\1^{\intercal} ) x  \ \ = \ \ \left(\1^\intercal x \right)^2 - n\cdot \sum_{i=1}^c \left(\1_{P_i}^\intercal x \right)^2. 
	\]
	Since $P_1, \hdots P_c$ partition the non-loops of $M$, $\1 $ equals $ \sum_{i=1}^c\1_{P_i}$. 
	For any real numbers $u_1, \hdots, u_c$, the Cauchy-Schwarz inequality implies that $(\sum_{i=1}^cu_i)^2 \leq c\cdot \sum_{i=1}^c u_i^2$. 
	This then gives that 
	\[
	\parens*{\1^\intercal x}^2 \  \ = \ \  \parens*{\sum_{i=1}^c\1_{P_i}^\intercal x}^2 \ \ \leq  \ \ c\cdot \sum_{i=1}^{c} \parens*{\1_{P_i}^\intercal x }^2 \ \  \leq  \ \ n\cdot \sum_{i=1}^c \parens*{\1_{P_i}^\intercal x}^2.
	\]
	For the last inequality, we use the fact the number of equivalence classes $c$ of nonloops of $M$ is at most $n$.
	It follows that $x^\intercal(n B - (n-1)\1\1^{\intercal} ) x \leq 0$ for all $x$ and by \cref{lem:LC_Quad}, $q$ is log-concave on $\R_{\geq 0}^{n+1}$. 
\end{proof}

\begin{proof}[Proof of \cref{thm:IndPolyCLC}]
	We will use the criterion in \cref{thm:CLCQuad} to show complete log-concavity. 
	
	Here we use $\partial_{i}$ to mean $\partial_{z_{i}}$ and for 
	$\alpha \in \Z_{\geq 0}^n$, $\partial^{\alpha}$ to denote $\prod_{i=1}^n\partial_i^{\alpha_i}$. 
	We need to show that for every $k\in \Z_{\geq 0}$ and $\alpha \in \Z_{\geq 0}^n$ with $k + \abs{\alpha}\leq n-2$, 
	the polynomial $\partial_y^{k}\partial^\alpha g_M$ is indecomposable and that for $k + \abs{\alpha} = n-2$ it is log-concave. 
	
	Note that if $\alpha_i\geq 2$ for any $i$, then $\partial^{\alpha}g_M$ is zero, so we may consider $\alpha = \1_J$ for some $J\subseteq [n]$. 
	Similarly, if $J$ is not an independent set of $M$, then $\partial^{\alpha}g_M = \partial^{J}g_M =0$. 
	Therefore is suffices to consider $\alpha = \1_J$ for $J\in \I$. In this case, 
	the derivative $\partial^{J}g_M$ equals the polynomial $g_{M/J}$ of the contraction $M/J$, namely 
	\[
	\partial^{J}g_M \ \ = \  \  \sum_{I\in \I : J\subseteq I}y^{n-\card{I}}\prod_{i\in I\setminus J} z_i
	\ \ = \  \  \sum_{I\in \I : J\subseteq I}y^{n-\card{J}-\card{I\setminus J}}\prod_{i\in I\setminus J} z_i \ \  = \ \ g_{M/J}.
	\] 
 Recall that $M/J$ is a matroid on ground set $[n]\setminus J$ with independent sets $\set{I\setminus J  \given J\subseteq I\in \I}$. 

First we check indecomposability of $\partial_y^{k}\partial^Jg_M =\partial_y^{k}g_{M/J}$. 
Note that if $i \in [n]\setminus J$ is a loop of $M/J$, then the variable $z_i$ does not appear in $g_{M/J}$ and $\partial_ig_{M/J} =0$. 
Similarly, $\partial_ig_{M/J}$ is zero for all $i\in J$. 
Otherwise, the monomial $y^{n-\card{J}-1-k}z_i$ appears in $\partial_y^{k}g_{M/J}$ with non-zero coefficient. 
 Since $k+\card{J} \leq n-2$, it follows that $\partial_y\partial_i g_{M/J}$ is non-zero. 
In particular, the graph formed in \cref{def:indecomposable} is a star centered at the variable $y$, and thus connected.
Therefore $\partial_y^{k}\partial^Jg_M$ is indecomposable. 

Now suppose $k + \card{J}=n-2$.  Since $M/J$ is a matroid on $n - \card{J}$ elements, \cref{lem:Rank2} imples that 
$\partial_y^{n-\card{J}-2}g_{M/J} = \partial_y^{k}\partial^Jg_{M}$ is log-concave on $\R_{\geq 0}^{n+1}$. All together with \cref{thm:CLCQuad}, this implies that the polynomial $g_M$ is completely log-concave. 
\end{proof}

\begin{corollary}\label{cor:BivariateIndep}
Given a matroid $M=([n], \I)$ with $\I_{k}$ independent sets of size $k$,  the bivariate polynomial
\[
f_M(y,z)\ = \ \sum_{k=0}^r  \ \I_{k}  \ y^{n-k}z^k, 
\]
is completely log-concave. 
\end{corollary}
\begin{proof}
Note that $f_M$ is the restriction of the completely log-concave polynomial $g_M$ to $z_i = z$ for all $i\in [n]$.  Since the image of $\R_{\geq 0}^2$ under the linear map 
$(y,z) \mapsto (y,z,\hdots, z)$ is contained in $\R_{\geq 0}^{n+1}$, \cref{lem:LinearPreserver} implies that 
 $f_M(y,z) = g_M(y,z, \hdots, z)$ is completely log-concave. 
\end{proof}

\section{Proof of Mason's conjecture}\label{sec:mason}

We use the following proposition, which was first observed by Gurvits \cite{Gur08}, and give a short proof for the sake of completeness. 

\begin{proposition}[Proposition 2.7 from \cite{Gur08}]\label{thm:CLCtoULC} 
If $f = \sum_{k=0}^{n} c_{k}y^{n-k}z^{k} \in \R[y,z]$ is completely log-concave, then 
the sequence $c_{0},\dots,c_{n}$ is ultra log-concave. That is, for every $1<k <n$,
\[\parens*{\frac{c_{k}}{\binom{n}{k}}}^{2} \geq \frac{c_{k-1}}{\binom{n}{k-1}} \cdot \frac{c_{k+1}}{\binom{n}{k+1}}.\]
\end{proposition}
\begin{remark}
In \cite{Gur08}, Gurvits assumes \emph{strong log-concavity} and also shows the
converse. In a future article, we show the equivalence of strong and complete log-concavity for 
homogeneous polynomials.
\end{remark}

\begin{proof}
 	Since $f$ is completely log-concave, for any $1< k<n$,  the quadratic 
	$q(y,z) = \partial_y^{n-k-1}\partial_z^{k-1} f$ is log-concave over $\R_{\geq 0}^2$. 
	Notice that for any $0\leq m \leq n$, 
	\[		
	 \partial_y^{n-m}\partial_z^{m} f \ = \ (n-m)!  \ m! \ c_{m} \ = \ n! \ \frac{c_{m}}{\binom{n}{m}}.
	\]	
	Using this for $m=k-1,k, k+1$, we can write the Hessian of $q$ as
	\[
	\nabla^2q  \ = \ 
	\begin{bmatrix} 
	 \partial_y^2 q	& \partial_y\partial_z q \\
	\partial_y\partial_z q 	& \partial_z^2 q
	 	\end{bmatrix}
	\ = \ n! \  
	\begin{bmatrix} 
	 \left. c_{k-1}  \middle/\binom{n}{k-1}\right.	& \left.c_{k}\middle/ \binom{n}{k}\right. \\
	 \left.c_{k} \middle/ \binom{n}{k}\right.  	& \left.c_{k+1} \middle/\binom{n}{k+1}\right.
	 	\end{bmatrix}.
	\]	
Since $q$ is log-concave on $\R_{\geq 0}^2$, by \cref{lem:LC_Quad} its Hessian cannot be positive or negative definite. 
Its determinant is therefore non-positive.  This gives the desired inequality: 
\[	
0  \ \geq  \ \det(\nabla^2q)  \ = \  (n!)^2\parens*{\frac{c_{k-1}}{ \binom{n}{k-1}} \cdot \frac{c_{k+1}}{ \binom{n}{k+1}} - \parens*{ \frac{c_{k}}{ \binom{n}{k}}}^2}. 
\]		
\end{proof}

The strong version of Mason's conjecture, \cref{thm:mason}, then follows from \cref{cor:BivariateIndep}.

\printbibliography
		
\end{document}